\documentclass[11pt]{amsart}
 
\theoremstyle{plain}
\newtheorem{thm}{Theorem}[section]
\newtheorem{theorem}[thm]{Theorem}

\newtheorem{lemma}[thm]{Lemma}

\newtheorem{proposition}[thm]{Proposition}

\numberwithin{equation}{section}

\title{Some approaches toward the Jacobian conjecture} 
\author{Tuyen Trung Truong}
\thanks{ The author was supported by Australian Research Council grants DP120104110 and DP150103442.}
\subjclass[2010]{}
\date{\today}
\address{School of Mathematical Sciences, The University of Adelaide, SA 5005, Australia}\email{tuyen.truong@adelaide.edu.au}
\begin{document}

\begin{abstract}
\noindent 
In this paper, we study a so-called Condition C1 and a weaker Condition C2. For Druzkowski maps Condition C2 is equivalent to the Jacobian conjecture. Main results obtained: 

- Stating new equivalent formulations of the Jacobian conjecture.

- Formulating some generalisations of the Jacobian conjecture  and giving both theoretical and experimental evidences to support them. 

- Showing Condition C1 holds for a generic matrix of any given rank, is an  invariant for a certain group action, and Condition C2 is an invariant for cubic similarity.  

- Giving one heuristic argument for the truth of the Jacobian Conjecture.

- Giving an effective (time saving) method to check whether a given Druzkowski map satisfies the Jacobian conjecture, explaining theoretically and checking on many examples including those previously considered by other authors.

- Proposing approaches toward resolving the Jacobian conjecture. Showing that a generic Druzkowski matrix satisfies the criteria of some of these approaches (see Theorem \ref{TheoremGeneric}), and hence expecting to be able to check these approaches for a given Druzkowski map very quickly. 

-As an application, proposing a strategy to use cubic similarity to check that Druzkowski maps of dimension $\leq 9$ satisfy the Jacobian conjecture. 
\end{abstract}

\maketitle


\section{Introduction}
The famous Jacobian Conjecture is the following statement:

{\bf Jacobian Conjecture.} Let $F=(F_1,\ldots ,F_n):\mathbb{C}^n\rightarrow \mathbb{C}^n$ be a polynomial map such that $JF$ (the Jacobian matrix $(\partial F_i/\partial x_j)_{1\leq i,j\leq n}$) is invertible at every point. Then, $F$ has a polynomial inverse. 

The Jacobian conjecture was first stated by Keller in 1939. Polynomial maps with invertible Jacobian matrices are called Keller maps. We denote by $JC(n)$ the Jacobian Conjecture in dimension $n$, and by $JC(\infty )$ the statement that $JC(n)$ holds for every $n$. In the literature $JC(\infty )$ is usually called the generalized Jacobian Conjecture. This conjecture has attracted a lot of works, and many partial results were found. For example,  Magnus - Applegate -Onishi -Nagata proved $JC(2)$ for $F=(P,Q)$ where the GCD of the degrees of $P,Q$ is either a prime number or $\leq 8$;  Moh proved $JC(2)$  for $\deg (F)\leq 100$; Wang proved that $JC(n)$ holds if $\deg (F)=2$; and Yu (also Chau and Nga) proved that if $F(X)-X$ has no linear term and has all non-positive coefficients then JC holds for $F$. For more details the readers can consult the reference list and the references therein. An excellent survey is the book \cite{essen}. We note that the $\mathbb{R}$-analog of $JC(2)$ (in this case, we require only that the map $F$ is bijective, since its inverse may not be a polynomial as the example $F(x)=x+x^3$ shows) is false, by the work of Pinchuk (Section 10 in \cite{essen}). 

There have been many reductions of the generalized Jacobian Conjecture $JC(\infty )$. One of these reductions is due to Bass, Connell, Wright and Yagzhev, who showed that to prove $JC(\infty )$, it is enough to prove for all   $F(x)=x+H(x)$ and all $n$, where $H(x)$ is a homogeneous polynomial of degree $3$ (Section 6.3 in \cite{essen}). Druzkowski made a further simplification (Section 6.3 in \cite{essen})  

\begin{thm}[Druzkowski]
$JC(\infty )$ is true if it is true for all the maps $F$ of the form $F(x)=(x_1+l_1(x)^3,\ldots ,x_n+l_n(x)^3)$ with invertible Jacobian $JF$, here $l_1,\ldots ,l_n$ are linear forms. 
\label{TheoremDruzkowski}\end{thm}
Later, Druzkowski \cite{druzkowski}) simplified even further showing that it is enough to show for the above maps with the additional condition that $A^2=0$, where $A$ is the $n\times n$ matrix whose $i$-th row is $l_i$. We then simply say that a matrix is Druzkowski if the corresponding map $F_A(x)=(x_1+l_1(x)^3,\ldots ,x_n+l_n(x)^3)$ is Keller, i.e. the determinant of its Jacobian is $1$.

A lot of efforts have been devoted to showing that the Druzkowski maps are polynomial automorphisms (Section 7.1 in \cite{essen} and for recent developments see  \cite{bondt-yan1,bondt-yan2,bondt-yan3}, and for a comprehensive reference on this topic see \cite{bondt}).  There are many partial results proved for this class of maps, for example it is known from the works of Druzkowski and Hubbers and many other people that the Druzkowski maps are polynomial automorphisms if either the rank of $A$ is $\leq 4$ or the corank of $A$ is $\leq 3$. In particular, the Jacobian conjecture was completely checked for Druzkowski maps in dimensions $\leq 8$ (see \cite{bondt}, also for stronger properties that can be proved for these polynomial automorphisms). Some new results on Druzkowski maps in dimension $9$ have been obtained recently in \cite{yan} and \cite{bondt-yan1,bondt-yan2,bondt-yan3}. A common theme of these proofs is that in these cases the Druzkowski maps are "equivalent" to some other polynomial maps for which the Jacobian conjecture is obviously true. It is not easy to see whether this strategy can work for higher ranks or coranks.

Despite these encouraging results, there has been little progress on checking the Jacobian conjecture for Druzkowski maps in higher dimensions. One of the difficulties for this lies in the fact that the structure of the Druzkowski maps is not well-understood, in the sense that for a large enough $n$ there is no easy way to produce all $n\times n$ Druzkowski matrices (for small values $n=3,4,5$, there are classifications by Meisters, Wright and Hubbers, see \cite{essen} and \cite{hubbers}). Because of this, even if we have already verified the Jacobian conjecture for the Druzkowski maps up to a certain dimension, say $n=10$, it is still hard to decide whether the Jacobian conjecture will hold for these maps in higher dimensions. A striking illustration of this undecidability is the following fact proved by Rusek \cite{rusek}, that the set of $n\times n$ Druzkowski matrices is not irreducible for $n\geq 530$. Similarly, some properties, which are known for some special classes of Druzkowski maps, do not hold for all of them. The same undecidability applies to the bigger class of Keller maps, since the structure of polynomial automorphisms (sought to be shown to be the same as the Keller maps) is not well-understood in dimensions $n\geq 3$. 

The main purpose of this paper is to propose some conditions (namely, Conditions 1 and 2 below) which are more amenable to using computer programs to check. The validity of these conditions for either a generic Druzkowski matrix or for all Druzkowski matrices with integer coefficients will prove the Jacobian conjecture. We also show, theoretically and via examples, that these conditions make the computations needed for the direct approach (that of showing Druzkowski maps are injective) a lot faster. Contrast to the case of dimension $2$, we note that the truth of the Jacobian conjecture in higher dimensions is not universally believed, the main reason is because of lack of enough evidence (for example the paper \cite{essen1} reported that the majority of participants of a conference in 1997 voted that the Jacobian conjecture is false). We hope that this paper will give some more support to the opposite conclusion and illustrate the idea that computers may be used in seeking of a solution to the Jacobian conjecture $JC(\infty )$. More precisely, the idea is to investigate the small dimensions using computer programs and then make an inductive guess. 

{\bf Heuristic argument for the truth of the Jacobian conjecture.} By the results proved later in this paper, it follows that Conditions C1 and C2 below hold for many square matrices, in particular for a dense set of all square matrices. Hence it is very reasonable that they are also true for a dense set of Druzkowski matrices, and the latter claim is enough for the truth of the Jacobian conjecture (see below). 

{\bf Remark on the effectiveness (time-saving) of our approaches.} In Theorem \ref{TheoremGeneric}  we will show that a generic Druzkowski matrix satisfies the criteria in Theorems \ref{TheoremFurtherReduction} and \ref{TheoremFurtherReduction1} below. Based on this we explain in Section 5.1 that we expect to be able to check whether a given Druzkowski matrix satisfies these criteria very quickly. This is a sign that our approaches seem very promising. 

The arguments in this paper are based, besides the Druzkowski reduction and Theorem \ref{TheoremFurtherReduction}  and Lemma \ref{Lemma1} to be stated below, on two other results on polynomial automorphisms. The first of these results belongs to Connell and van den Dries (Proposition 1.1.19 in \cite{essen}) :
\begin{proposition}
If for some $n\geq 2$ there is a counter example to $JC(n)$ then there is a counter example to $JC(N)$ (for some $N>n$) with integer coefficients.  
\label{PropositionConnellDries}\end{proposition}
Applying this proposition, in order to prove the Jacobian conjecture for every dimensions, it suffices to do so for polynomials with integer coefficients. Then, by looking at the proof of the reductions by Bass, Connell, Wright, Yagzhev and Druzkowski (\cite{essen}), it is clear that we need to consider only Druzkowski matrices with rational coefficients. Then using some properties of these matrices (see the next section of this paper), we are reduced to consider only Druzkowski matrices with integer coefficients. 

The second of the results mentioned above is an invertibility criterion for polynomial maps using derivations (Section 3.1 in \cite{essen}). Here we briefly recall  this criterion. Let $F=(F_1,\ldots ,F_n):\mathbb{C}^n\rightarrow \mathbb{C}^n$ be a polynomial map such that $F(0)=0$  and $\det (JF)=1$. We then define related derivations by the formula
\begin{eqnarray*}
(\frac{\partial}{\partial F_1},\ldots ,\frac{\partial}{\partial F_n})^t:=((JF)^{-1})^{t}(\frac{\partial}{\partial x_1},\ldots ,\frac{\partial}{\partial x_n})^t.
\end{eqnarray*}  
(Here $(.)^t$ means the transpose of a matrix.) Now we add $n$ new variables $y_1,\ldots ,y_n$ and consider the derivation
\begin{eqnarray*}
D_F:=y_1\frac{\partial}{\partial F_1}+\ldots +y_n\frac{\partial}{\partial F_n}.
\end{eqnarray*}
 Let $d=\deg (F)$ and put $N=d^{n-1}+1$. We have (Proposition 3.1.4 in \cite{essen}) 
 \begin{proposition}
 $F$ has a polynomial inverse if and only if $D_F^{N}x_i=0$ for all $i=1,\ldots ,n$.
 \label{PropositionDensity}\end{proposition}
 
Applying this proposition, we see that to prove the Jacobian conjecture in every dimensions, it suffices to do so for a dense set of Druzkowski maps.  We briefly describe the proof of this claim. From the definition of $D_F$, we see that it is continuous on the set of $F:\mathbb{C}^n\rightarrow \mathbb{C}^n$ with a fixed degree $d$ (for Druzkowski maps $d=3$) and $\det (JF)=1$. Here we identify such a polynomial with its coefficients, and the topology is the usual one on the corresponding affine space. Hence if $F_j\rightarrow F$ with $\det (JF_j)=1$ for all $j$ and $D_{F_j}^Nx_i=0$ for all $j$ then $D_F^Nx_i$ is also zero.  
 
 We summarise these as the following
 \begin{theorem}
 For the Jacobian conjecture to hold in every dimensions, it suffices to either 
 
 i) Show that all Druzkowski maps whose matrix has integral coefficients satisfy the Jacobian conjecture,
 
 or 
 
 ii) Show that a dense set of Druzkowski maps  satisfies the Jacobian conjecture. 
  \label{TheoremDensityInteger}\end{theorem}
 
 {\bf Remarks.} While the reduction to integer coefficients have been studied quite extensively (e.g. Section 10.3 in \cite{essen}, and also for the related topic of Jacobian conjecture in positive characteristics), it seems that the reduction to a dense set of maps has not been widely discussed. (Note that in \cite{cima-gasull-llibre-manosas}, a good dense subset of all real polynomial maps whose Jacobian is invertible everywhere was identified.) Here we illustrate one simple use of its.  Yan \cite{yan} showed that Druzkowski maps whose matrices $A=(a_{ij})_{1\leq i,j\leq n}$ are of dimension $n\leq 9$ and $\prod _{i=1}^{n}a_{i,i}\not= 0$ satisfy the Jacobian conjecture. Hence if we can show that this set is dense in the set of Druzkowski maps of dimension $\leq 9$ then the Jacobian conjecture is true for all Druzkowski maps of dimension $\leq 9$. We will discuss more on this in Section 5. 
\subsection{Main results}
Let us fix some notations to be used throughout the paper.

{\bf Notations.} We will use the following notations. For  vectors $u,v\in \mathbb{C}^n$, we define $u*v:=(u_1v_1,\ldots ,u_nv_n)$ (coordinate-wise multiplication) and $u^k=u*u*\ldots *u$ ($k$-th self-multiplication of $u$), and we define by $\Delta [u]$ the diagonal $n\times n$ matrix whose $(i,i)$-th entry is $u_i$. Thus, the Druzkowski maps and their Jacobians can be written as
\begin{eqnarray*}
F(x)&=&x+(A.x)^3,\\
JH(x)&=&Id+3\Delta [(A.x)^2].A.
\end{eqnarray*} 

For a fix $n$, we let $\mathcal{M}_n$ denote the set of all $n\times n$ matrices with coefficients in $\mathbb{C}$. We also use the following notations: 
\begin{eqnarray*}
V_n&:=&\{(y,z,A)\in \mathbb{C}^n\times \mathbb{C}^n\times \mathcal{M}_n:~~\det (Id+\Delta [(sz+ty)^2].A)=1 ~\forall s,t\in \mathbb{C},\\&&~z+A.(z^3+z*y^2)=0\},\\
W_n&:=&\{(y,z,A)\in \mathbb{C}^n\times \mathbb{C}^n\times \mathcal{M}_n:~\det (Id+\Delta [(sz+tA.y)^2].A)=1 \forall s,t\in \mathbb{C},\\
&&~z+A.(z^3+z*(A.y)^2)=0\}.
\end{eqnarray*}
Note that while these two varieties look very similar, they are different in that in $V_n$ only $y$ appears in the equations and in $W_n$ only $A.y$ appears in the equations. Note also that for a Druzkowski matrix $A$, then in $W_n$ we do not need to check the condition $\det (Id+\Delta [(sz+tA.y)^2].A)=1$.

The starting point of our work is the following, to be derived in Section 2.
\begin{theorem}
Let $A$ be an $n\times n$ matrix and $F_A(x)=x+(A.x)^3~:\mathbb{C}^n\rightarrow \mathbb{C}^n$  the corresponding cubic linear map. Then $F_A$ is an automorphism if and only if the solutions to $z+A.(z^3)+A.(z*(A.y)^2)=0$ are exactly $z=0$.
\label{TheoremInjectiveCriterion}\end{theorem}

Based on this result, we are now ready to state the two conditions. 

{\bf Condition C1.}  An $n\times n$  matrix $A$ satisfies Condition C1 if whenever $(y,z,A)\in V_n$ then $z=0$.

{\bf Condition C2.} An $n\times n$  matrix $A$ satisfies Condition C2 if whenever $(y,z,A)\in W_n$ then $z=0$.

Condition C2 is clearly weaker than Condition C1, and for Druzkowski matrices it will be seen later that Condition C2 and the Jacobian conjecture are equivalent. From both theoretical and experimental considerations, it is reasonable to conjecture that all Druzkowski matrices satisfy Condition 1 (see the Main Conjecture in the next section for more details). Also we conjecture that all $n\times n$ matrices with integer entries satisfy Conditions 1 and 2. Even stronger, for support see in particular Theorem \ref{TheoremCubicSimilarity}, we conjecture that all matrices satisfy Condition 2.  

We first observe a simple way to construct more matrices satisfying Conditions 1 and 2. 
\begin{theorem}
Let $A$ be of the block form 
\[ \left( \begin{array}{cc}
A_{1,1}&A_{1,2}\\
0&0
\end{array}\right) \]
where $A_{1,1}$ is a square matrix. Then $A$ satisfies Condition 1 (or Condition 2) if and only if $A_{1,1}$ is so. 
\label{TheoremReduction}\end{theorem}

For Condition 2 only, we show that it is satisfied for a matrix $A$ iff it is satisfied for all matrices which are cubic similar to $A$. Cubic similarity has been used in the study of Druzkowski matrices, see e.g. \cite{hubbers}. We recall that two $n\times n$ matrices $A$ and $B$ are called cubic similar if there is an invertible matrix $L$ such that $(Bx)^3=L^{-1}(ALx)^3$ for all $x\in \mathbb{C}^n$.
\begin{theorem}
A matrix $A$ satisfies Condition 2 iff any matrix $B$ cubic similar to it also satisfies Condition 2. 
\label{TheoremCubicSimilarity}\end{theorem}

We recall that given an $n\times n$ matrix of $A$, a principal minor of $A$ is the determinant of a matrix obtained from $A$ by deleting $k$ rows $i_1,\ldots ,i_k$ and $k$ columns $i_1,\ldots ,i_k$ (here the indices of the rows and columns are the same).  
\begin{theorem}
Condition C1 holds for an $n\times n$ matrix $A$ in the following cases: 

1) $n=2,3$;

or 

2) $A$ has rank $1$;

or

3) $A$ is upper or lower triangular;

or 

4) All the principal minors of $A$ are non-zero. In particular, the set of matrices $A$ in $\mathcal{M}_n$ for which Condition C1 holds is {dense} in $\mathcal{M}_n$. 
\label{Theorem1}\end{theorem}

Part 4) of the above Theorem shows that Condition C1 holds for a generic square matrix of any dimension. However, these matrices are invertible, and hence are of not much interest to the question of whether the Druzkowski maps satisfy the Jacobian conjecture (this is because the Druzkowski matrices all have determinant $0$). We will show that given any positive integer $r$, Condition C1 is also satisfied for a generic matrix of rank $r$ (see Section 3). This claim is proved using the following property of the varieties $V_n$ and $W_n$. 

\begin{theorem}
1) Let $(y,z,A)$ be in $V_n$ and $D$ any invertible diagonal $n\times n$ matrix. Then $(Dy,Dz,DAD^{-3})$ is also in $V_n$. 

2) Let $(y,z,A)$ be in $W_n$ and $D$ any invertible diagonal $n\times n$ matrix. Then  $(D^3y,Dz,DAD^{-3})$ is also in $W_n$. 
\label{TheoremGroupAction}\end{theorem}

Hence, there is an action of the group of invertible diagonal matrices on the sets of matrices (not) satisfying Conditions C1 and C2, given by $A\mapsto DAD^{-3}$. (In fact, assume for example that $A$ does not satisfy Condition 1. Then,  there are $y,z$ with $z\not= 0$ such that $(y,z,A)\in V_n$. Then for  any invertible matrix $D$ we also have $(Dy,Dz,DAD^{-3})\in V_n$ with $Dz\not= 0$. Therefore, $DAD^{-3}$ does not satisfy Condition 1 either.) 

By the proof of Theorem \ref{TheoremGroupAction}, if $A$ is a Druzkowski matrix and $D$ is an invertible diagonal matrix, then $DAD^{-3}$ is also a Druzkowski matrix. (Remark: This is a special case of the cubic similarity mentioned above, with $B=DAD^{-3}$ and $L=D^3$.) Combining this theorem and Lemma \ref{Lemma}, we obtain the following useful equivalent formulation of the Jacobian conjecture.
\begin{theorem}
To prove the Jacobian conjecture $JC(\infty )$, it is necessary and sufficient to show that for any $n$ and any $k\in \{3,\ldots ,n\}$, there are no Druzkowski $n\times n$ matrix $A$ and $y\in \mathbb{C}^n$  such that
\begin{eqnarray*}
Z_k+A.(Z_k^3+Z_k*(A.y)^2)=0,
\end{eqnarray*}
where $Z_k=(1,\ldots ,1,0,\ldots ,0)^t\in \mathbb{C}^n$ is the vector whose first $k$ coordinates are $1$ and whose last $n-k$ coordinates are $0$.
\label{TheoremFurtherReduction}\end{theorem}  
In particular, $JC(\infty )$ holds if the analog of the above criterion for Condition 1 is true.
\begin{theorem}
To prove the Jacobian conjecture $JC(\infty )$, it is sufficient to show that for any $n$ and any $k\in \{3,\ldots ,n\}$, there are no Druzkowski $n\times n$ matrix $A$ and $y\in \mathbb{C}^n$  such that $(y,Z_k,A)\in V_n$ where $Z_k=(1,\ldots ,1,0,\ldots ,0)^t\in \mathbb{C}^n$ is the vector whose first $k$ coordinates are $1$ and whose last $n-k$ coordinates are $0$.
\label{TheoremFurtherReduction1}\end{theorem}  
In Section 4 we will illustrate the effectiveness (i.e. time saving) of using these two theorems and Conditions 1 and 2 on specific examples. In Section 5 we will explain theoretically this effectiveness and compare our approaches to several existing approaches. To conclude this subsection, we observe that a generic Druzkowski map satisfies the criteria in these two theorems. 
\begin{theorem}
a) A generic Druzkowski map satisfies the criteria in Theorems \ref{TheoremFurtherReduction} and \ref{TheoremFurtherReduction1}. More precisely, let $A$ be a generic Druzkowski matrix of dimension $n$. Then, there are no $k\in \{3,\ldots ,n\}$ and $y\in \mathbb{C}^n$ such that 
\begin{eqnarray*}
\det (Id+\Delta [(sz+ty)^2].A)&=&1,~\forall s,t\in \mathbb{C},\\
Z_k+A.(Z_k^3+Z_k*y^2)&=&0.
\end{eqnarray*}

b) More generally, if $V$ is a subvariety of $\mathcal{M}_n$ invariant under the action of the group of invertible  diagonal matrices, i.e. $DVD^{-3}=V$ for all invertible diagonal matrix $D$, then a generic element of $V$ satisfies the criteria in Theorems \ref{TheoremFurtherReduction} and \ref{TheoremFurtherReduction1}.
\label{TheoremGeneric}\end{theorem}

\subsection{A geometric approach}

Since $V_n$ (and $W_n$) are algebraic subvarieties of an affine space, it follows that $V_n$ (respectively $W_n$) have a finite number of irreducible components. It turns out that each component is either contained in or disjoint from $\{z=0\}$ . 
\begin{theorem}
Fix $n\geq 2$. Let $V$ be an irreducible component of $V_n$ or $W_n$. Then $V\subset \{z=0\}$ iff $V\cap \{z=0\}\not=\emptyset$. The same result holds more generally for connected components of $V_n$ or $W_n$.
\label{TheoremConnectedness}\end{theorem} 
Rusek's result \cite{rusek} showed that the "geometric approach" of showing that the set of Druzkowski matrices of a given dimension $n$ is irreducible does not work in dimension $\geq 530$. However, from Theorem \ref{TheoremConnectedness}, we can propose another "geometric approach" consisting in showing that $W_n$ is connected. (Note that from Theorem \ref{TheoremConnectedness} and Section 4.3, it follows that $V_n$ is not connected for $n\geq 4$.) In fact, we can state a geometric equivalent formulation of the Jacobian Conjecture.
\begin{theorem}
$JC(\infty )$ holds iff for all $n\in \mathbb{N}$ the affine variety $W_n':=\{(y,z,A)\in W_n$ and $A$ is a Druzkowski matrix $\}$ is connected.

In particular, if $W_n$ is connected for every $n\in \mathbb{N}$ then $JC(\infty )$ holds.  
\label{TheoremEquivalentFormulation}\end{theorem}
 \subsection{Organization of the paper} 
The remaining of this paper is organized as follows. In Section 2 we explain how Conditions 1 and 2 are derived, and why they imply the Jacobian Conjecture. We state there one generalization of the Jacobian conjecture. In Section 3, we prove the above theorems. In Section 4, we present the computational experimental computations, including those on some Druzkowski matrices previously considered by other authors. In Section 5, we give details of the approaches together with detailed strategies to employ them. In the same section we also explain theoretically the effectiveness (time saving) of our approaches in practice, compare them with other existing approaches, and state some generalisations of the Jacobian conjecture together with both theoretical and experimental evidences to support them. As an application, we suggest a strategy to use cubic similarity to check that all Druzkowski maps in dimension $\leq 9$ satisfy the Jacobian conjecture. In the Appendix we present the Mathematica codes used. 

{\bf Acknowledgements.} We are benefited from discussions with Neeraj Kashyap and Hang Tien Nguyen on computational aspects. The latter also helped with running some examples. Laughlin Andrew Campbell brought the paper \cite{cima-gasull-llibre-manosas} to our attention. The comments and suggestions of them and Finnur Larusson, Eric Bedford and Tien-Cuong Dinh helped to improve the paper. Most of the experimental computations that require long time and big memory have been done on the BigMem cluster of the Tizzard super computer of eResearch South Australia, and Andrew Hill there generously helped us. Part of the research was carried out while the author was at the Korea Institute for Advanced Study (KIAS), and we were helped by the KIAS Center for Advanced Computation (with computer resources) and Hoyoung Kim (with technical issues). John Dixon and Craig Bauling from the Wolfram company replied to our enquiries, in particular instructed us to a Mathematica command which  describes the Groebner Basis in terms of the original defining polynomials. We would like to thank all these institutions and people for their valuable help.

\section{Derivation of the Conditions}

In this section we explain the derivation of the Conditions 1 and 2 and  show why they imply the Jacobian Conjecture.  

There is a well-known result that a polynomial self-map of $\mathbb{C}^n$ is an automorphism if it is injective (Chapter 3 in \cite{essen}).  In theory, we can check, for each given dimension $n$, whether all Druzkowski maps are injective by using a computer program (for example Mathematica) to find the Groebner basis for the ideal $\mathcal{I}$ defined by the equations $x+(A.x)^3=y+(A.y)^3$ and
\begin{equation}
\det (Id+\Delta [(A.z)^2].A)=1
\label{Equation0}\end{equation}
 for all $z\in \mathbb{C}^n$,  to see that $x-y$ belongs to this Groebner basis. However, in practice one faces the difficulty that the number of the polynomials in the ideal defined by the system $\det (Id+\Delta [(A.z)^2].A)=1$ for all $z\in \mathbb{C}^n$ grows very fast with respect to the dimension $n$: it is roughly the same as the number of monomials of degrees at most $n$ in $n$ variables. We note that an explicit procedure for writing down these equations was given in \cite{gorni-gasinska-zampieri}.
 
 {\bf Remark.} However, we note from the experiments (Section 4 in this paper), that in practice it is quite effective to use Theorems \ref{TheoremGroupAction} and \ref{TheoremFurtherReduction} to check that any given Druzkowski map satisfies the Jacobian conjecture with the help of a computer program. 

This paper grew out of the author's curiosity to see whether we can reduce the number of equations defining the ideal $\mathcal{I}$. (It is a classical result, \cite{eisenbud-evans} and references therein, that any algebraic subvariety of $\mathbb{C}^N$ is defined by $N$ equations, however it is quite challenging to find the equations for explicit examples.) It also originated from our trying to ponder on the following question: 

{\bf Question.} If a formal proof of the Jacobian conjecture is to be found for all Druzkowski maps of degree $3$ in all dimensions $n$, how can we make use of the assumption that $JF$ is invertible? 

To be more explicit about this Question, let us first make some simple algebraic reductions. Let $u,v\in \mathbb{C}^n$ be such that $F(u)=F(v)$, that is $u+(A.u)^3=v+(A.v)^3$.  Then by subtracting and using that $A$ is a linear map, we find that 
$$(u-v)+(A.u-A.v)*( (A.u)^2+(A.u)*(A.v)+(A.v)^2 )=0.$$ 
If we define $x=u-v$ then we can write the above equation as
\begin{eqnarray*}
x+(A.x)*( (A.x)^2+3(A.x)*(A.v)+3(A.v)^2 )=0.
\end{eqnarray*}  
Now, by substituting
$$y=\sqrt{3}v+\frac{\sqrt{3}}{2}x,$$
and then replacing $x$ by $x/2$, we see that the above equation is reduced to 
\begin{equation}
x+(A.x)*((A.x)^2+(A.y)^2)=0.
\label{Equation1}\end{equation} 
Then, the fact that $F$ is injective is the same as that the equation $x+(A.x)*((A.x)^2+(A.y)^2)=0$ has only the solution $x=0$. This and Lemma \ref{Lemma1} below lead to Condition 2. 

Now, we see that $y$ appears in the ideal $\mathcal{I}$ only through $A.y$. Hence it is natural to ask whether the Jacobian conjecture is in fact stronger, that is in Equation (\ref{Equation1}), we can replace $Ay$ by $y$ (which of course must satisfy a condition compatible with Equation (\ref{Equation0})) and still obtain the same conclusion? Hence, we state a weaker version of our main conjecture:

{\bf Conjecture (Weaker version).} Let $A$ be a Druzkowski $n\times n$ matrix. Assume that $y\in \mathbb{C}^n$ satisfy
\begin{eqnarray*}
\det (Id + \Delta [(A.x +ty)^2].A)=1
\end{eqnarray*} 
for all $t\in \mathbb{C}$ and all $x\in \mathbb{C}^n$. Then, if $x+(Ax)*((Ax)^2+y^2)=0$, we must have $x=0$. 

{\bf Remark.} If $A$ is an $n\times n$ matrix for which $\det (Id + \Delta [(A.x +ty)^2].A)=1$ for all $t\in \mathbb{C}$ and $x\in \mathbb{C}^n$ then $A$ must be a Druzkowski map as we can readily see by putting $t=0$ in the equality. Hence the above Conjecture, while a bit stronger than the Jacobian conjecture, is only for Druzkowski matrices. 

We may push this investigation further, by asking that in showing that $x+(A.x)*((A.x)^2+y^2)=0$ has only the solution $x=0$, do we need the assumption (\ref{Equation0}) somehow on the plane generated by $A.x$ and $y$ only? This leads us to state the formulation of our main conjecture. 
 
 {\bf Main Conjecture.} Let $A$ be a Druzkowski $n\times n$ matrix. Assume that $x,y\in \mathbb{C}^n$ satisfy
\begin{eqnarray*}
\det (Id + \Delta [(sA.x+ty)^2].A)=1
\end{eqnarray*} 
for all $s,t\in \mathbb{C}$. Then, if moreover $x+(Ax)*((Ax)^2+y^2)=0$, we must have $x=0$. 

These two conjectures can be seen to be more general than the original Jacobian conjecture. If we ask for not only Druzkowski matrices but general $n\times n$ matrices and use the following Lemma, we arrive at Condition 1. 
\begin{lemma}
The following two statements are equivalent:

1) There is a non-zero solution $x$ to $x+(A.x)^3+(A.x)*y^2=0$,

and

2) There is a non-zero solution $z$ to $z+A.(z^3+z*y^2)=0$.
\label{Lemma1}\end{lemma}
\begin{proof}
($\Rightarrow$) If $x$ is a non-zero solution to $x+(A.x)^3+(A.x)*y^2=0$ then $z=A.x$ is non-zero. Moreover, we have 
\begin{eqnarray*}
0&=&A(x+(A.x)^3+(A.x)*y^2)=A(x+z^3+z*y^2)\\
&=&A(x)+A(z^3+z*y^2)=z+A(z^3+z*y^2).
\end{eqnarray*}

($\Leftarrow$) If $z$ is a non-zero solution to $z+A(z^3+z*y^2)=0$, by defining $x=-(z^3+z^*y^2)$ we see that $Ax=z$. In particular, $x$ is also non-zero. Moreover,
\begin{eqnarray*}
0=x+z^3+z*y^2=x+(A.x)^3+(A.x)*y^2.
\end{eqnarray*}
\end{proof}

In Section 4 we provide experimental evidences to support these Conjectures. In Section 5.1 we give theoretical reasons to support these Conjectures. 

\section{General properties}
In the first subsection of this, we will prove the results in the introduction. In the second subsection, we show that Condition 1 is satisfied for a generic matrix of any given rank. 
\subsection{Proofs of the theorems in the Introduction}
\begin{proof}[Proof of Theorem \ref{TheoremReduction}]
We prove for example for Condition 1. By induction on $k$, we may assume that $k=n-1$. We write $z=(z',z_n)$ and $y=(y',y_n)$ where $z',y'\in \mathbb{C}^{n-1}$. The proof is completed provided we can show the following: $z_n=0$, and $(y,z,A)$ is in $V_n$ if and only if $(y',z',A_{1,1})$ is in $V_{n-1}$. In fact, look at the last equation in $z+A.(z^3+z*y^2)=0$, we find that $z_n=0$. Then the first $n-1$ equations reduce to $z'+A_{1,1}(z'^3+z'*y'^2)=0$.  We can check conversely that if $z'+A_{1,1}(z'^3+z'*y'^2)=0$  and $z_n=0$ then $z+A.(z^3+z*y^2)=0$. 

Now we consider the condition $\det (Id +\Delta [(sz+ty)^2].A)=1$ for all $s,t\in \mathbb{C}$. This condition is equivalent to that $\Delta [(sz+ty)^2].A$ is nilpotent. The matrix $\Delta [(sz+ty)^2].A$ has the block form 
\[ \left( \begin{array}{cc}
\Delta [(sz'+ty')^2].A_{1,1}&\Delta [(sz'+ty')^2].A_{1,2}\\
0&0
\end{array}\right) \]
From this we can check easily that $\Delta [(sz+ty)^2].A$ is nilpotent if and only if $\Delta [(sz'+ty')^2].A_{1,1}$ is nilpotent. Since the latter is satisfied for all $s,t\in \mathbb{C}$, it is equivalent to $\det (Id+\Delta [(sz'+ty')^2].A)=1$ for all $s,t\in \mathbb{C}$. 
\end{proof}
\begin{proof}[Proof of Theorem \ref{TheoremCubicSimilarity}]
A matrix $A$ satisfies Condition 2 iff whenever $x,y\in \mathbb{C}^n$ are such that
\begin{eqnarray*}
\det (Id +\Delta [(sAx+tAy)^2].A)&=&1,~\forall s,t,\\
x+(Ax)^3+(Ay)^2*(Ax)&=&0,
\end{eqnarray*}
then $x=0$. 

Now assume that $A$ satisfies Condition 2. Let $B$ be a matrix which is cubic similar to $A$. We will show that $B$ also satisfies Condition 2. First, the assumption that $B$ is cubic similar to $A$ implies the existence of an invertible matrix $L$ such that $(Bx)^3=L^{-1}(ALx)^3$ for all $x\in \mathbb{C}^n$. Computing the Jacobian we find that 
\begin{eqnarray*}
\Delta [(Bx)^2].B=L^{-1}.\Delta [(ALx)^2].AL
\end{eqnarray*}
for all $x\in \mathbb{C}^n$.

Now assume that $x,y\in \mathbb{C}^n$ are such that
 \begin{eqnarray*}
\det (Id +\Delta [(sBx+tBy)^2].B)&=&1,~\forall s,t,\\
x+(Bx)^3+(By)^2*(Bx)&=&0.
\end{eqnarray*}
We need to show that $x=0$. 

We look first at the determinant condition. We have, for all $s,t$
\begin{eqnarray*}
1&=&\det (Id +\Delta [(sBx+tBy)^2].B)=\det (Id+L^{-1}\Delta [(sALx+ALy)^2].AL)\\
&=&\det (Id+\Delta [(sALx+tALy)^2].A).
\end{eqnarray*}

Next we look at the system of $n$ cubic equations in $x$ and $y$. By multiplying with $L$ we have
\begin{eqnarray*}
0=Lx+L(Bx)^3+L.(By)^2*Bx=Lx+(ALx)^3+(ALy)^2*(ALx).
\end{eqnarray*} 

Therefore, $x'=Lx$ and $y'=Ly$ sastisfy
\begin{eqnarray*}
\det (Id +\Delta [(sAx'+tAy')^2].A)&=&1,~\forall s,t,\\
x'+(Ax')^3+(Ay')^2*(Ax')&=&0.
\end{eqnarray*}
Since $A$ satisfies Condition 2, it follows that $x'=0$, and hence $x=L^{-1}x'=0$ as wanted. 
\end{proof}

\begin{proof}[Proof of Theorem \ref{TheoremConnectedness}]
We prove for example 1). Let $V$ be an irreducible component of  $V_n$. We need to show that if $V\cap \{z=0\}\not=\emptyset$ then $V\subset \{z=0\}$. Assume that there is $(y,0,A)\in V$ and a sequence $(y^{(j)},z^{(j)},A^{(j)})\in V$ such that
\begin{eqnarray*}
z^{(j)}&\not=&0,~\forall j,\\
(y^{(j)},z^{(j)},A^{(j)})&\rightarrow&(y,0,A).
\end{eqnarray*} 
We will show a contradiction. 

We define, as in the proof of Lemma \ref{Lemma1}, 
$$x^{(j)}=-z^{(j)}*z^{(j)}*z^{(j)}-z^{(j)}*y^{(j)}*y^{(j)}.$$ 
Then $x^{(j)}\not= 0$ for all $j$ and $x^{(j)}\rightarrow 0$. Moreover, $x^{(j)}+(A^{(j)}x^{(j)})^3+(A^{(j)}x^{(j)}).y^{(j)}*y^{(j)}=0$ for all $j$. We can rewrite this equation as
\begin{eqnarray*}
(Id+A^{(j)}.\Delta [(A^{(j)}.y^{(j)})^2]).x^{(j)}=-(A^{(j)}.x^{(j)})^3.
\end{eqnarray*}
Since $(y^{(j)},z^{(j)},A^{(j)})\in V_n$, it follows that $\det (Id+A^{(j)}.\Delta [(A^{(j)}.y^{(j)})^2])=1$ for all $j$ (see the proof of Lemma \ref{Lemma} for more details). The fact that $(y^{(j)},z^{(j)},A^{(j)})$ converges to $(y,0,A)$ implies that the inverse matrices $(Id+A^{(j)}.\Delta [(A^{(j)}.y^{(j)})^2])^{-1}$ are bounded. From
\begin{eqnarray*}
x^{(j)}=-(Id+A^{(j)}.\Delta [(A^{(j)}y^{(j)})^2])^{-1}.(A^{(j)}x^{(j)})^3,
\end{eqnarray*}
it follows that $||x^{(j)}||\leq C||x^{(j)}||^3$ for some positive constant independent of $j$. The assumption that $x^{(j)}\rightarrow 0$ then implies that $x^{(j)}=0$ for large $j$, as wanted.
\end{proof}
\begin{proof}[Proof of Theorem \ref{TheoremGroupAction}] We prove for example 1). Let $(y,z,A)$ be in $V_n$ and $D$ an invertible diagonal matrix. We need to show that $(y',z',A')=(Dy,Dz,DAD^{-3})$ is also in $V_n$. 

First, using that $x'^3=D^3x^3$ and similarly $x'*y'^2=D^3x*y^2$ since $D$ is a diagonal matrix, we have
\begin{eqnarray*}
x'+A'.(x'^3+x'*y'^3)=Dx+DAD^{-3}(D^3x^3+D^3x*y^2)=D.(x+A.(x^3+x*y^2))=D.(0)=0.
\end{eqnarray*}

It remains to check that $\det (Id+\Delta [(sx'+ty')^2].A')=1$ for all $s,t\in \mathbb{C}$. We note that since $D$ is a diagonal matrix
\begin{eqnarray*}
\Delta [(sx'+ty')^2]=\Delta [(sx+ty)^2].D^2=D^2.\Delta [(sx+ty)^2],
\end{eqnarray*}
and hence
\begin{eqnarray*}
\det (Id+\Delta [(sx'+ty')^2].A')&=&\det (Id+\Delta [(sx+ty)^2].D^2.D.A.D^{-3})\\
&=&\det (Id+D^3.\Delta [(sx+ty)^2].A.D^{-3})\\
&=&\det (Id+\Delta [(sx+ty)^2].A)\\
&=&1
\end{eqnarray*}
for all $s,t\in \mathbb{C}$, as wanted. 
\end{proof}
\begin{proof}[Proof of Theorem \ref{Theorem1}]
1) The proof of this case will be given in the next section with the help of computer programs.

2) Assume that $A$ has rank $1$. Let $A_j$ denote the $j$-th row of $A$. Without loss of generality, we may assume that $A_1\not=0$ and $A_j=\lambda _jA_1$ for some $\lambda _j\in \mathbb{C}$ ($j=2,\ldots ,n$). For convenience, we define $\lambda _1=1$.

The equation $\det (Id+\Delta [(sz+ty)^2].A)=1$ for all $s,t\in \mathbb{C}$ becomes
\begin{eqnarray*}
\sum _{i=1}^{n}\lambda _iz_i^2a_{1,i}&=&0,\\
\sum _{i=1}^{n}\lambda _iy_i^2a_{1,i}&=&0,\\
\sum _{i=1}^{n}\lambda _iz_iy_ia_{1,i}&=&0.
\end{eqnarray*}
The equation $z+A.(z^3+z*y^2)=0$ becomes $z_i=\lambda _iz_1$ ($i=1,\ldots ,n$) together with
\begin{eqnarray*}
z_1+\sum _{i=1}^n(z_i^3a_{1,i}+z_iy_i^2a_{1,i})=0.
\end{eqnarray*}
Substituting $z_i=\lambda _iz_1$ into other equations, we obtain
\begin{eqnarray*}
z_1^2\sum _{i=1}^{n}\lambda _i^3a_{1,i}&=&0,\\
\sum _{i=1}^{n}\lambda _iy_i^2a_{1,i}&=&0,\\
z_1\sum _{i=1}^{n}\lambda _i^2y_ia_{1,i}&=&0,\\
z_1+z_1^3\sum _{i=1}^n\lambda _i^3a_{1,i}+z_i\sum _{i=1}^n\lambda _iy_i^2a_{1,i}&=&0.
\end{eqnarray*}
The first, third and fourth equations imply that $z_1=0$ and hence $z_i=0$ for all $i$.
 
3) We may assume that $A$ is upper triangular. The equation $\det (Id+\Delta [(sz+ty)^2].A)=1$ for all $s,t\in \mathbb{C}$ becomes
\begin{eqnarray*}
(1+(sz_1+ty_1)^2a_{1,1})\ldots (1+(sz_n+ty_n)^2a_{n,n})=1,~\forall s,t\in \mathbb{C}.
\end{eqnarray*}
From this, it follows that for all $i$, either $a_{i,i}=0$ or $y_i=z_i=0$. Solving the equation $z+A.(z^3+z*y^2)=0$ from bottom up, we then see that all $z_i$ are $0$.  

4) This is proven in the next subsection. 
\end{proof}
\begin{proof}[Proof of Theorem \ref{TheoremEquivalentFormulation}]
If $JC(n)$ holds then $W_n'=\{(y,0,A):~y\in \mathbb{C}^n$ and $A$ is a Druzkowski matrix$\}$, and hence is connected since the set of Druzkowski matrices is connected. (If $A$ is a Druzkowski matrix then $tA$ is also a Druzkowski matrix for any $t\in \mathbb{C}$. In particular, there is a path connecting $A$ and $0$.)

If $W_n'$ is connected, then Theorem \ref{TheoremConnectedness} (or rather, its proof) shows that $W_n'\subset \{z=0\}$ and hence $JC(n)$ holds. 

If $W_n$ is connected, then $W_n\subset \{z=0\}$, and hence so is $W_n'$.
\end{proof}

\begin{proof}[Proof of Theorem \ref{TheoremGeneric}] We give only the proof of a), since the proof of b) is identical. 

We will use the arguments and notations of the next subsection. Let $A$ be a Druzkowski matrix. Then the set consisting of all matrices of the form $DAD^{-3}$, where $D$ runs over all invertible diagonal matrices,  belongs to the same irreducible component of all Druzkowski matrices. (In fact, let $f$ be the map from the set of invertible diagonal matrices to the set of Druzkowski matrices defined by $D\mapsto DAD^{-3}$. This is a regular morphism between algebraic varieties. Since the set of invertible diagonal matrices is irreducible, it follows that there is an irreducible component $W$ of the set of Druzkowski matrices for which $f^{-1}(W)$ is the whole of invertible diagonal matrices.) We need to show only that at least one among these matrices satisfy the criteria in Theorem \ref{TheoremFurtherReduction1}. Assume otherwise. Then, in particular $A$ does not satisfy the criteria in Theorem \ref{TheoremFurtherReduction1} with respect to some $k\in \{3,\ldots ,n\}$.

Let $A_{1,1}$ be the $k\times k$ submatrix of $A$ as in the next subsection.  Then the arguments in the next section shows that $A_{1,1}$ is nilpotent, in particular is of rank $<k$. We will show that for a generic choice of the invertible diagonal matrix $D$, then $(1,\ldots ,1)^t$ does not belong to the image of $D_1A_{1,1}D_1^{-3}$ where $D_1$ is the $k\times k$ submatrix of $D$ as in the next subsection. Therefore, for such  a choice of $D$, there is no $y$ for which $(y,z=Z_k,DAD^{-3})\in V_n$, as wanted. Assume that this is not the case, we will deduce a contradiction. 

In fact, assume that for all invertible diagonal matrix $D$ then $(1,\ldots ,1)^t$ belongs to the image of $D_1A_{1,1}D_1^{-3}$. Then we see that $D_1^{-1}(1,\ldots ,1)^t$ belongs to the image of $A_{1,1}D_1^{-3}$ and hence to the image of $A_{1,1}$ for all such $D$. But the set of all such vectors $D_1^{-1}.(1,\ldots ,1)^t$ is exactly the set $\{(x_1,\ldots ,x_k)\in \mathbb{C}^k:~x_1\ldots x_k\not= 0\}$. Since $A_{1,1}$ is nilpotent, it cannot contain all of this set. This gives a contradiction as desired. 

Then the intersection of all these generic sets, when $k$ runs over all the set $\{3,\ldots ,n\}$, is still a generic set. All matrices in this intersection set satisfies the criterion in Theorem \ref{TheoremFurtherReduction1} for all $k\in \{3,\ldots ,n\}$. 

Finally, using the properties of the projections of affine algebraic varieties (in particular, the Closure Theorem in Section 6, Chapter 5 in \cite{cox-little-shea}), we conclude that there is a proper subvariety (and moreover does not contain any irreducible component) of the set of all Druzkowski matrices  outside which the criteria in Theorem \ref{TheoremFurtherReduction1} hold. 
\end{proof}

\subsection{More matrices satisfying Condition 1}
Applying Theorem \ref{TheoremGroupAction}, we can reduce the study of $V_n$ to a simpler case as follows. Let $(y,z,A)$ be in $V_n$. Let us define $w=(w_1,\ldots ,w_n)^t$ where $w_i=1$ if $z_i=0$, and $w_i=1/z_i$ otherwise. Then the diagonal matrix $D=\Delta [w]$ is invertible, $(y',z',A')=(Dy,Dz,DAD^{-3})$ is also in $V_n$ and $z'^2=z'$. 

Fix $r>0$ a positive integer. Let us choose $(y,z,A)$ an element in $V_n$ such that $z^2=z$ and $A$ is of rank $r$. After a permutation, we can assume that $z=(1,\ldots ,1,0,\ldots ,0)^t$ has the first $k$ entries to be $1$ and the last $n-k$ entries to be $0$. We write $A$ in the block form 
\[ \left( \begin{array}{cc}
A_{1,1}&A_{1,2}\\
A_{2,1}&A_{2,2}
\end{array}\right) \]
where $A_{1,1}$ is of the size $k\times k$. The set $\mathcal{E}_r$ of $n\times n$ matrices for which all minors up to dimension $r$ are non-zero is dense in the set of all matrices of rank $r$, hence we can consider only these matrices. In the condition $\det (Id +\Delta [(sz+ty)^2])=1$ for all $s,t\in \mathbb{C}$, if we choose $t=0$ we see that $\Delta [z^2].A$ is nilpotent. Then the fact that $z=(1,\ldots ,1,0,\ldots ,0)^t$ implies that $A_{1,1}$ is a nilpotent $k\times k$ matrix. Since $A\in \mathcal{E}_r$, it follows that $k\geq r$, and $A_{1,1}$ has rank exactly $r$. Next, the condition that $z+A.(z^3+z*y^2)=0$ implies in particular that $(1,\ldots ,1)^t$ is in the image of $A_{1,1}$. If $k\geq 2$ then the set of all $k\times k$ matrices $A_{1,1}$ satisfying the above two conditions is a very small set (more specifically, of high codimension) in the set of all $k\times k$ matrices of rank $r$. Here is a sketch of the proof for this claim. 
\begin{proof}
We note that since $A_{1,1}$ has rank $r$, the requirement that $A_{1,1}$ is nilpotent is described by $r$ equations, coming from that $A_{1,1}$ is nilpotent iff $A_{1,1}|_V:V\rightarrow V$ is nilpotent, where $V=$ image of $A_{1,1}$ is of dimension $r$. 

Now the condition that $(1,\ldots ,1)^t$ is in the image of $A_{1,1}$ is described by $k-r$ equations. 

All of the above equations are homogeneous in the entries of $A_{1,1}$. Hence the set of such matrices are defined by $k$ homogeneous equations in the entries of $A_{1,1}$. Then we check that these equations in fact define a codimension $k$ subvariety. 
\end{proof}
Let $D$ be an invertible diagonal $n\times n$ matrix. If $D$ has the block form
\[ \left( \begin{array}{cc}
D_1&0\\
0&D_2
\end{array}\right) \]
where $D_1$ is of size $k\times k$, then $DAD^{-3}$ has the block form
\[ \left( \begin{array}{cc}
D_1A_{1,1}D_1^{-3}&D_1A_{1,2}D_2^{-3}\\
D_2A_{2,1}D_1^{-3}&D_2A_{2,2}D_2^{-3}
\end{array}\right) \]
Hence, the orbit of all such matrices $\Lambda$ under the action of the group of diagonal matrices $D$ by $A\mapsto DAD^{-3}$, plus the permutations, is also very small. Here is  a sketch of the proof for this claim. 
\begin{proof}
In fact, since $\Lambda$  is defined by $k$ homogeneous equations and the set of all diagonal matrices $D_1$ is of dimension $k$, the total dimension of the orbit of $\Lambda$ is only $=\dim (\Lambda )+k-1$ (and not $\dim (\Lambda )+k$). Hence the orbit of $\Lambda$ is of codimension $1$ in the set of all $k\times k$ matrices of rank $r$. 
\end{proof}
The complement $\Gamma$, which satisfies Condition 1, is therefore big, and is dense in the set of matrices of a given rank $r$.  Here we illustrate the situation when $n=2$. In this case, by the same argument as that of Lemma \ref{Lemma} below, the complement of $\Gamma$ is the set of all $2\times 2$ matrices $A$ of the form $A=DA_0D^{-3}$, where $D$ is an invertible diagonal matrix and $A_0$ is the matrix
\[ \left( \begin{array}{cc}
1&-1\\
1&-1
\end{array}\right) \]
Hence the complement of $\Gamma$ is only of dimension $2$, and since the set of $2\times 2$ matrices of rank at most $1$ has dimension $3$, we see that $\Gamma$ is dense in the latter. 

We note that if $A$ is any $n\times n$ matrix whose every principal minor is non-zero, then the same is true for $DAD^{-3}$ for any invertible diagonal $n\times n$ matrix $D$. Therefore, in this case if $k>0$ then the $A_{1,1}$ in the above cannot be nilpotent. Thus we obtain a proof for part 4) of Theorem \ref{Theorem1}. 

The next Lemma deals with the remaining case $k=1$.
\begin{lemma}
Let $k=1$ or $2$. There is no $(y,z,A)$ in $V_n$ with $z=(1,\ldots ,1,0,\ldots ,0)$ where the first $k$ entries are $1$ and the last $n-k$ entries are $0$.  
\label{Lemma}\end{lemma} 
\begin{proof}
We first consider the case $k=1$.  Assume that there is $(y,z,A)$ in $V_n$ where $z=(1,0,\ldots ,0)$. Then, there is $s\in \mathbb{C}$ such that $(sz+y)^2=z^2+y^2$. From this we have
\begin{eqnarray*}
0&=&z+A.(z^3+z*y^2)=(Id+A.\Delta [z^2+y^2]).z\\
&=&(Id+A.\Delta [(sz+y)^2]).z
\end{eqnarray*}
which will imply $z=0$ provided that $Id +A.\Delta [(sz+y)^2]$ is invertible. To this end, it suffices to show that $\det (Id +A.\Delta [(sz+ty)^2])=1$ for all $s,t\in \mathbb{C}$. In fact, since $(y,z,A)$ is in $V_n$, we have that $\det (Id +D(s,t).A)=0$ for all $s,t\in \mathbb{C}$, where $D(s,t)=\Delta [(sz+ty)^2]$. The latter is the same as $D(s,t).A$ is nilpotent, that is $(D(s,t).A)^n=0$. Then, 
\begin{eqnarray*}
(A.D(s,t))^{n+1}=A.(D(s,t).A)^n.D(s,t)=0,
\end{eqnarray*}
which implies that $A.D(s,t)$ is also nilpotent for every $s,t$. This then implies that $\det (Id+A.D (s,t))=1$ for all $s,t$ as wanted. 

It remains to consider the case $k=2$. In this case, we write $A$ in the block form as
\[ \left( \begin{array}{cc}
A_{1,1}&A_{1,2}\\
A_{2,1}&A_{2,2}
\end{array}\right) \]
where 
\[ A_{1,1}=\left( \begin{array}{cc}
a&b\\
c&d
\end{array}\right) \]
is a $2\times 2$ matrix. We have as before $A_{1,1}^2=0$ and 
\begin{eqnarray*}
(1,1)^t+A_{1,1}.(1+y_{1}^2+y_2^2)^t=0.
\end{eqnarray*}
Multiplying the above system with $A_{1,1}$, using that $A_{1,1}^2=0$, we find that 
\begin{eqnarray*}
c+d=b-d=a+d=-1+d(y_1^2-y_2^2)=0.
\end{eqnarray*}

Now consider again the condition $\det (Id+\Delta [(sz+ty)^2])=0$ for all $s,t\in \mathbb{C}$. Expanding the left hand side as a polynomial in variables $s,t$, we have that its homogeneous part of degree $2$, which is $$a_{1,1}(s+ty_1)^2+a_{2,2}(s+ty_2)^2+a_{3,3}(ty_3)^2+\ldots +a_{n,n}(ty_n)^2,$$ must be $0$. In particular, consider the coefficient of the term $st$, we find that $a_{1,1}y_1+a_{2,2}y_2=-d(y_1-y_2)$ must be $0$. This contradicts the condition $-1+d(y_1^2-y_2^2)=0$ we found in the above. 
\end{proof}

The conclusion of Lemma \ref{Lemma} does not hold for the case $k=3$ and bigger, see Section 4. For Druzkowski maps, see however Sections 2 and 5.1.
\section{Results proven with the help of computer programs}

In this section we present the results we obtained with the help of computer programs. We have used two main resources: a Mathematica software run on a normal personal MacBook Air laptop and a MuPad software run on the BigMem cluster  of the Tizzard super computer of eResearch SA. For the MuPad computations, the typical configuration is  $1$ node whose memory size is about $120GB$ plus $100 GB$ virtual, and the working duration is about $4$ to $5$ days. We compute the Groebner basis of the corresponding polynomial systems, and (except Section 4.3) look to see whether $z_1,\ldots ,z_n$ appear in the Groebner Basis. 

{\bf Remark.} We will consider here Conditions 1 and 2 with an additional requirement $\det (A)=0$, because the main interest is in Druzkowski maps for which this condition is obviously satisfied. 
\subsection{The case $n=2$}
In this case, since $\det (A)=0$ we have that $A$ has rank $1$. Hence Condition 1 is satisfied.  
\subsection{The case $n=3$}
Using Mathematica, we find that Condition 1 is true for $3\times 3$ matrices whose determinant is $0$. Thus part 1 of Theorem \ref{Theorem1} is proved. 
 \subsection{The case $n=4$, rank $=2$}
For the case $n=4$, the computation requires so long time and big memory that it does not terminate on a personal computer. Hence we have to use MuPad on the super computer. We found there is $(y,z,A)$ in $V_4$ (recall the notations $V_n$ and $W_n$ from the Introduction) such that $z=(1,1,1,0)$ and $A$ has rank $2$. The time to compute the Groebner Basis was 9911 seconds, and the Groebner Basis has $179$ elements. The Groebner Basis is too complicated (it takes more than $70$ pages to print out) to extract any useful information at the moment. In particular, we cannot conclude whether there is such a counterexample with integer coefficients. 
 
Using MuPad, we check that there is no $(y,z,A)$ in $W_4$ such that $z=(1,1,1,0)$ and $A$ has rank $2$. However, the check with $z=(1,1,1,1)$ could not terminate, hence the situation for Condition 2 is still unclear for us. See, however, Section 5.1.
\subsection{Random matrices with integer coefficients}
Before we proved the results in the Subsections 3.2 and 4.3, in a previous version we used Mathematica to investigate Condition 1 on  randomly generated $4\times 4$ matrices with integer coefficients in the interval $[-25,25]$. On all of those examples we found that Condition 1 is satisfied and the Groebner Basis has $6$ elements.  

From the results in Section 3.2, we know that for a generic matrix with integer coefficients then Condition 1 is satisfied. It is still open whether Condition 1 is satisfied by all matrices with integer coefficients. 

We have used MuPad to investigate higher dimensions and ranks. Below is a summary, the matrices here are randomly generated with integer coefficients lying between $0$ and $10^{12}$: 

-  $n=4$, rank $=3$: The Groebner Basis has $8$ elements, time to compute it is $16$ seconds.

-  $n=5$, rank $=3$: The Groebner Basis has $9$ elements, time to compute it is $358$ seconds.

-  $n=5$, rank $=4$: The Groebner Basis has $15$ elements, time to compute it is $15641$ seconds.

For $n=10$, rank $=3$ or $n=7$, rank $=4$ the computations are usually terminated because of running out of time or memory, even when we restrict the entries to a smaller range. 
\subsection{Some examples of Druzkowski matrices}

We present here the experiments with some examples of Druzkowski matrices previously considered by other authors. 

{\bf Example 1.} This example is taken from \cite{gorni-zampieri}  (also page 140 in \cite{essen}), where Gorni and Zampieri developed their pairing between cubic homogeneous maps and cubic linear maps (see the citations for more details). In this example, $A$ is the following $15\times 15$ matrix 
\[ \left( \begin{array}{ccccccccccccccc}
0&0&0&0&0&0&0&0&0&0&0&0&0&0&0\\
0&0&0&0&0&0&0&0&0&0&0&0&0&0&0\\
0&0&0&-2&-1&1&1&1&0&0&-1&0&0&-1&0\\
0&0&-1&0&-1&0&1/2&0&0&1/2&0&-1/2&-1/2&0&0\\
0&0&1&-2&0&0&0&1&-1&-1&-1&0&0&0&1\\
1&0&1&-2&0&0&0&1&-1&-1&-1&0&0&0&1\\
0&1&1&-2&0&0&0&1&-1&-1&-1&0&0&0&1\\
1&0&-1&0&-1&0&1/2&0&0&1/2&0&-1/2&-1/2&0&0\\
1&0&0&-2&-1&1&1&1&0&0&-1&0&0&-1&0\\
0&1&0&-2&-1&1&1&1&0&0&-1&0&0&-1&0\\
1&0&1&0&1&0&-1/2&0&0&-1/2&0&1/2&1/2&0&0\\
0&1&-1&2&0&0&0&-1&1&1&1&0&0&0&-1\\
0&1&0&2&1&-1&-1&-1&0&0&1&0&0&1&0\\
1&1&1&-2&0&0&0&1&-1&-1&-1&0&0&0&1\\
1&1&0&-2&-1&1&1&1&0&0&-1&0&0&-1&0\\
\end{array}\right) \]
This is a Druzkowski matrix and the corresponding Keller map satisfies the so-called globally analytically linearisable condition, whose conjugations are polynomial maps. We check that $A^2=0$, $A$ is of rank $5$ and it has some non-zero $2\times 2$ principal minor. 

Using Mathematica, we check that this matrix $A$ satisfies Condition 1. The computation of the Groebner basis takes $9.62$ seconds. We observe that when trying to check  that the corresponding Druzkowski map is injective, an interesting phenomenon occurs. If we use the more obvious condition, that is $x+(Ax)^3+(Ax)*(Ay)^2=0$ implies $x=0$, then under the same setting (i.e. the same choice of monomial ordering as in the computation for Condition 1) it takes a very long time ($610.90$ seconds) to compute the Groebner basis. However, if we use the transformation in Lemma \ref{Lemma1}, that is $z+A.(z^3+z*(Ay)^2)=0$ implies $z=0$, it takes only $0.361$ seconds to compute the Groebner basis.  

{\bf Example 2.} This example is taken from page 197 in \cite{essen}. Here $A$ is the following $17\times 17$ matrix
\begin{eqnarray*}
&&[0,0,0,1/6,1/6,-1/3,-1/6,-1/6,1/3,0,0,0,0,0,0,0,1]\\
&&[0,0,0,{1}/{6},{1}/{6},-{1}/{3},-{1}/{6},-{1}/{6},{1}/{3},0,0,0,0,0,0,0,-1]\\
&&[0,0,0,1/6,1/6,-1/3,-1/6,-1/6,1/3,0,0,0,0,0,0,0,0]\\
&&[1/6,1/6,-1/3,0,0,0,0,0,0,0,0,0,0,0,0,0,1]\\
&&[-1/6,-1/6,1/3,0,0,0,0,0,0,0,0,0,0,0,0,0,1]\\
&&[0,0,0,0,0,0,0,0,0,0,0,0,0,0,0,0,1]\\
&&[0,0,0,0,0,0,0,0,0,-1/6,-1/6,1/3,1/12,1/12,-1/12,-1/12,1]\\
&&[0,0,0,0,0,0,0,0,0,-1/6,-1/6,1/3,1/12,1/12,-1/12,-1/12,-1]\\
&&[0,0,0,0,0,0,0,0,0,-1/6,-1/6,1/3,1/12,1/12,-1/12,-1/12,0]\\
&&[1/6,-1/6,-1/6,0,0,0,0,0,0,0,0,0,0,0,0,0,1]\\
&&[1/6,-1/6,-1/6,0,0,0,0,0,0,0,0,0,0,0,0,0,-1]\\
&&[1/6,-1/6,-1/6,0,0,0,0,0,0,0,0,0,0,0,0,0,0]\\
&&[1/6,1/6,-1/3,1/6,1/6,-1/3,-1/6,-1/6,1/3,0,0,0,0,0,0,0,1]\\
&&[1/6,1/6,-1/3,-1/6,-1/6,1/3,1/6,1/6,-1/3,0,0,0,0,0,0,0,-1]\\
&&[1/6,1/6,-1/3,1/6,1/6,-1/3,-1/6,-1/6,1/3,0,0,0,0,0,0,0,-1]\\
&&[1/6,1/6,-1/3,-1/6,-1/6,1/3,1/6,1/6,-1/3,0,0,0,0,0,0,0,1]\\
&&[0,0,0,0,0,0,0,0,0,0,0,0,0,0,0,0,0]
\end{eqnarray*}
It is a Druzkowski matrix of rank $5$. While $A^2\not= 0$, $A^3=0$. The corresponding Keller map is a counterexample to the cubic-linear globally linearisation condition, mentioned in Example 1.

For this example, we do not know whether it satisfies Condition 1 or not (see, however, Section 5.1). Checking the Jacobian conjecture on this example is still very quick. Computing the Groebner Basis for the system $x+(Ax)^3+(Ax)*(Ay)^2=0$ takes $0.577$ seconds, and computing the Groebner Basis for the system $z+A.(z^3+z*(Ay)^2)=0$ takes $0.337$ seconds. 

One difference between this example and other examples considered in this subsection is that it has some non-zero principal $3\times 3$ minors. We may speculate that because of this the computation of the Groebner Basis for Condition 1 takes a longer time. We also remark that the criterion in Theorem \ref{TheoremFurtherReduction1} is satisfied on this example, and the time needed to compute it is quite fast. For example, with $k=17$ then the time needed to compute the Groebner Basis for the polynomial system in Theorem \ref{TheoremFurtherReduction1} is only $0.17979$ seconds. This is what to be expected from Theorem \ref{TheoremGeneric}.

{\bf Example 3.} This example is taken from the paper \cite{adamus-bogdan-hajto}, where the authors proposed an approach toward the Jacobian conjecture. Here $A$ is the $13\times 13$ matrix 
\[ \left( \begin{array}{ccccccccccccc}
0&0&0&1/6&1/6&-1/3&-1/6&-1/6&1/3&0&0&0&1\\
0&0&0&1/6&1/6&-1/3&-1/6&-1/6&1/3&0&0&0&-1\\
0&0&0&1/6&1/6&-1/3&-1/6&-1/6&1/3&0&0&0&0\\
1/6&1/6&-1/3&0&0&0&0&0&0&0&0&1&0\\
1/6&1/6&-1/3&0&0&0&0&0&0&0&0&-1&0\\
1/6&1/6&-1/3&0&0&0&0&0&0&0&0&0&0\\
0&0&-1/3&0&0&0&0&0&0&1/6&1/6&0&1\\
0&0&-1/3&0&0&0&0&0&0&1/6&1/6&0&-1\\
0&0&-1/3&0&0&0&0&0&0&1/6&1/6&0&0\\
0&0&0&1/6&1/6&-1/3&-1/6&-1/6&1/3&0&0&1&0\\
0&0&0&1/6&1/6&-1/3&-1/6&-1/6&1/3&0&0&-1&0\\
0&0&0&0&0&0&0&0&0&0&0&0&0\\
0&0&0&0&0&0&0&0&0&0&0&0&0\\
\end{array}\right) \]
This is a Druzkowski matrix of rank $5$ and satisfies $A^2=0$. For this example, the method in \cite{adamus-bogdan-hajto} consists of showing that certain $1170$ Wronskians belong to a certain ring. The computation is quite involved and is contained in a big PDF file on the authors' website \cite{adamus-bogdan-hajto2}. 

We have checked by Mathematica that this matrix satisfies Condition 1, and the computation of the Groebner Basis takes $0.137$ seconds. Computing the Groebner Basis for the system $x+(Ax)^3+(Ax)*(Ay)^2=0$ takes $0.535$ seconds, and computing the Groebner Basis for the system $z+A.(z^3+z*(Ay)^2)=0$ takes $0.116$ seconds. 

{\bf Example 4.} This example is taken from the paper \cite{adamus-bogdan-crespo}, where the authors show that their previous results can be used to prove that some Druzkowski maps fulfilling a certain nilpotency condition satisfy the Jacobian conjecture. Here $A$ is the $5\times 5$ matrix
\[ \left( \begin{array}{ccccc}
0&a_2&a_3&a_4&a_5\\
0&0&b_3&b_4&b_5\\
0&0&0&0&0\\
0&0&0&0&0\\
0&0&0&0&0\\
\end{array}\right) \]
Since $A$ is upper triangular, it satisfies Condition 1 by Theorem \ref{Theorem1}.  

{\bf Example 5.} This example is taken from \cite{yan}. Here $A$ is the $4\times 4$ matrix
\[ \left( \begin{array}{cccc}
1&i&1&1\\
-i&1&-i&-i\\
-1&-i&1&-1\\
-1&-i&1&-1\\
\end{array}\right) \]
This is a Druzkowski matrix of rank $2$. Its characteristic polynomial is $t^2(t^2-2t+4)$, hence in particular it is not nilpotent. 

We have checked by Mathematica that this matrix satisfies Condition 1, and the computation of the Groebner Basis takes $0.0135$ seconds. Computing the Groebner Basis for the system $x+(Ax)^3+(Ax)*(Ay)^2=0$ takes $0.0138$ seconds, and computing the Groebner Basis for the system $z+A.(z^3+z*(Ay)^2)=0$ takes $0.0123$ seconds. 

{\bf Example 6.} This example is taken from \cite{gorni-zampieri}. Here $A$ is the following $16\times 16$ matrix
\begin{eqnarray*}
&&[0,0,0,0,0,0,0,-1/3,1/6,1/6,1/24,-1/24,-1/24,1/24,0,0]\\
&&[0,0,0,0,0,0,0,-1/3,1/6,1/6,1/24,-1/24,-1/24,1/24,0,-1]\\
&&[1/3,-1/6,-1/24,1/24,-1/6,1/24,-1/24,0,0,0,0,0,0,0,1,-1]\\
&&[1/3,-1/6,-1/24,1/24,-1/6,1/24,-1/24,0,0,0,0,0,0,0,-1,-1]\\
&&[0,0,0,0,0,0,0,-1/3,1/6,1/6,1/24,-1/24,-1/24,1/24,0,1]\\
&&[1/3,-1/6,-1/24,1/24,-1/6,1/24,-1/24,0,0,0,0,0,0,0,1,1]\\
&&[1/3,-1/6,-1/24,1/24,-1/6,1/24,-1/24,0,0,0,0,0,0,0,-1,1]\\
&&[1/3,-1/6,-1/24,1/24,-1/6,1/24,-1/24,0,0,0,0,0,0,0,0,0]\\
&&[1/3,-1/6,-1/24,1/24,-1/6,1/24,-1/24,0,0,0,0,0,0,0,1,0]\\
&&[1/3,-1/6,-1/24,1/24,-1/6,1/24,-1/24,0,0,0,0,0,0,0,-1,0]\\
&&[0,0,0,0,0,0,0,-1/3,1/6,1/6,1/24,-1/24,-1/24,1/24,1,-1]\\
&&[0,0,0,0,0,0,0,-1/3,1/6,1/6,1/24,-1/24,-1/24,1/24,-1,-1]\\
&&[0,0,0,0,0,0,0,-1/3,1/6,1/6,1/24,-1/24,-1/24,1/24,-1,-1]\\
&&[0,0,0,0,0,0,0,-1/3,1/6,1/6,1/24,-1/24,-1/24,1/24,1,1]\\
&&[0,0,0,0,0,0,0,-1/3,1/6,1/6,1/24,-1/24,-1/24,1/24,-1,1]\\
&&[0,0,0,0,0,0,0,0,0,0,0,0,0,0,0,1]\\
&&[0,0,0,0,0,0,0,0,0,0,0,0,0,0,0,0]\\
\end{eqnarray*}
It is a Druzkowski matrix of rank $4$. As in Example 2, while $A^2\not= 0$, $A^3=0$. The corresponding Keller map still satisfies the cubic-linear globally linearisation condition, mentioned in Examples  1 and 2. However, in contrast to Example 1, here the conjugations are not polynomials.

We have checked by Mathematica that this matrix satisfies Condition 1, and the computation of the Groebner Basis takes $0.418$ seconds. Computing the Groebner Basis for the system $x+(Ax)^3+(Ax)*(Ay)^2=0$ takes $3.198$ seconds, and computing the Groebner Basis for the system $z+A.(z^3+z*(Ay)^2)=0$ takes $0.557$ seconds. In this example, in contrast to the previous ones, computing the Groebner Basis for Condition 1 is faster than that of the system $z+A.(z^3+z*(Ay)^2)=0$, but the time difference is not quite significant.  

{\bf Remark.} When we apply Theorems \ref{TheoremFurtherReduction} and \ref{TheoremFurtherReduction1} to these examples, the computation of the Groebner Bases is also very quick. Describing the Groebner Basis (which is $\{1\}$ in these cases) in terms of the original defining polynomials is also very quick. In fact, for some values of $k$, in some of these systems $1$ already appears as one of the original defining polynomials. 

\section{The  approaches and conclusions}

We propose in this section several approaches towards resolving the Jacobian conjecture and detailed strategies for implementing them. We recall here the notation used in the statement of Theorem \ref{TheoremFurtherReduction}: given $n$ and $k\in \{3,\ldots ,n\}$, we denote by $Z_k=(1,\ldots ,1,0,\ldots ,0)\in \mathbb{C}^n$ the vector whose first $k$ coordinates are $1$ and whose last $n-k$ coordinates are $0$. We start with some comments on the effectiveness of the approaches proposed and the comparisons to several existing approaches of other authors. 

\subsection{The effectiveness of Conditions 1 and 2 and comparisons to existing approaches}  In this subsection we explain theoretically why the Conditions 1 and 2, and Theorems  \ref{TheoremInjectiveCriterion}, \ref{TheoremFurtherReduction} and \ref{TheoremFurtherReduction1} are quite effective in practice. We also compare our approach to several existing approaches. 

In many of the existing approaches, a common theme is to construct the inverse of a given Keller map. Usually, this is  achieved by writing out a formula for the inverse and giving a bound for the degree of the inverse. The bound is then translated into a corresponding system of polynomials. For example, the derivation criterion Proposition \ref{PropositionDensity} is of this nature, and the bound there is $d^{n-1}$ which grows exponentially. For a Druzkowski map in dimension $n$ this bound is $3^{n-1}$ which is quite big even if $n$ is small. There are recent approaches (for example \cite{adamus-bogdan-hajto}) which reduce the complexity, however as illustrated in Example 3 in Section 4.5, in practice the computations following this "finding inverse" approach are still quite large. Since these approaches aim to apply for all polynomial automorphisms (a very big and largely unknown set) and to find inverses of these (a very difficult task), it can be expected from simple speculations that they will require a lot of time and efforts when applied to real situations.

Our approach is to check instead only the injectivity and does not try to construct the inverse map. For a general map, checking whether it is injective may be not any easier than finding its inverse. However, this turns out to be very suitable for Druzkowski maps, because of special properties of Linear Algebra. Here, given an $n\times n$ matrix $A$,  we only need to check that a system of $n$ equations of degree $3$ in $2n$ variables $y_1,\ldots ,y_n$ and $z_1,\ldots ,z_n$ has a very simple solution set, meaning $\{z_1=\ldots =z_n=0\}$. Moreover, the application of the simple transformation in Lemma \ref{Lemma1} turns out to be very essential in order to save time in practical situations. The work we need to do is reduced and since the solution set is very simple, we expect that it is quite effective to use in practice. The results in Section 4.5 illustrate this point.  (Note that previously Yaghzev also discussed the injectivety of a cubic homogeneous polynomial map in \cite{yaghzev}. However, since the class of cubic homogeneous polynomials is still very big, there were not as many simplifications as in the case of cubic linear maps.)  

In the same vent (and actually with more evidences as will be shown), we can also see that applying Theorems \ref{TheoremFurtherReduction} and \ref{TheoremFurtherReduction1} are quite effective in practice. In fact given a matrix $A$ and a special value $z=Z_k$, the system we need to check for Theorem \ref{TheoremFurtherReduction} consists of only $n$ equations (and it is $n^2+n$ equations for Theorem \ref{TheoremFurtherReduction1}) in $n$ variables $y_1,\ldots ,y_n$.  The solution set we look for is the simplest possible, that is the empty set. Moreover, a generic Druzkowski matrix satisfies the criteria in these theorem automatically, as shown in Theorem \ref{TheoremGeneric}. All these points come together to help us to solve these systems very quickly in practice, as illustrated in Section 4.5. 

 Here we note one other feature, that of the generality of the approaches. Most of the existing approaches are for Keller maps only, because they consist in constructing inverse maps. So while these approaches are useful in that they not only check that a given Keller map has an inverse but also construct that inverse, they are not intended for maps which are not invertible. The Mathieu subspaces approach in the paper \cite{derksen-essen-zhao}  is one of the rare generalisations of the Jacobian conjecture which still holds until now. Our approaches here point out some more possible generalisations of the Jacobian conjecture: one is that Condition 1 is true for all Druzkowski matrices, one other is that Conditions 1 and 2 are true for all matrices with integer coefficients and yet another one is that $W_n$ is connected. Like in the paper \cite{derksen-essen-zhao}, we show that Condition 1 is true for many maps that are not polynomial automorphisms. Moreover, we show that Condition 1 is true for a dense set of square matrices, and this enables us to propose a heuristic argument for the truth of the Jacobian conjecture. Finally, we point out one essential difference between the Mathieu subspaces approach in \cite{derksen-essen-zhao} and ours. The Mathieu subspaces approach is probabilistic in nature, hence seems not easy to employ computer programs to study (and there is no such attempt as far as we know). Our approach is algebraic in nature, and hence is very easy to employ computer programs. 
 
{\bf Remark.} Besides the experimental evidences for the truth of the generalisations of the Jacobian conjecture mentioned above, here we give some theoretical evidences for them. 

First, for the conjecture that all Druzkowski matrices should satisfy Condition 1. The first evidence for it is that we derived Condition 1 from the assumption that the Jacobian conjecture is true for Druzkowski matrices and the observation that $y$ appears in the polynomial equations defining the corresponding system only through $Ay$. Moreover, we showed that Condition 1 is satisfied by many $n\times n$ matrices, and is invariant under the action of invertible diagonal matrices via $A\mapsto D.A.D^{-3}$. We note also that Hubbers \cite{hubbers} showed that any Druzkowski matrix in dimension $5$ is cubic similar to an upper triangular matrix, and hence we expect that Condition 1 should be at least true for Druzkowski matrices up to dimension $5$. (Note that, unlike Condition 2, Condition 1 is not known to be an invariant for cubic similarity.)  

Second, for the conjecture that all $n\times n$ matrices should satisfy Condition 2. We know that Condition 1, and hence also Condition 2, is satisfied for many $n\times n$ matrices. Moreover, we know that Condition 2 is invariant under cubic similarity. The latter property enlarges the set of matrices satisfying Condition 2. 

 \subsection{Approach 0: Checking Theorems \ref{TheoremFurtherReduction} and \ref{TheoremFurtherReduction1} with a generic Druzkowski map} This is the approach that we find most practical and is based on the arguments in Section 5.1. The idea is as follows. As pointed out in Section 5.1, a generic Druzkowski map satisfies the criteria in Theorem \ref{TheoremGeneric}, hence we expect to be able to check very quickly that the Groebner Basis for the system $Z_k+A.(Z_k^3+Z_k*(A.y)^2)=0$ (or even the stronger system $Z_k+A.(Z_k^3+Z_k*y^2)=0$) is $\{1\}$. Then, we will be able to write $1$ in terms of the $n$ polynomials defining the system. If we have done this for a big enough number of Druzkowski matrices in a certain dimension $n$, we will be able to write down (using extrapolation) $1$ in terms of the $n$ polynomial for a {\bf general} Druzkowski matrix of that given dimension $n$. 
 
 The difficulty here is that we do not know much about Druzkowski matrices in higher dimensions. However, we do not really need to know all Druzkowski matrices in a certain dimension. What we really need is to be able to generate a (finite) "good" set of such Druzkowski matrices. Then, for each of Druzkowski matrices in this finite set we expect to be able to check very quickly. This is very promising, for example we can generate many such maps by starting from some classes of polynomial automorphisms (there are many good such ones, for example the tame class) and then going through the procedures given by Bass, Connell, Wright, Yagzhev and Druzkowsk to produce Druzkowski matrices. 
 
\subsection{Approach 1: Directly checking Condition 1 or the Jacobian conjecture for Druzkowski matrices}  This is similar to the Approach 0, the difference is that we consider all Druzkowski matrices instead of only generic ones. Here is the strategy we propose to make use of the help of computer programs in settling the Jacobian conjecture. The Groebner Basis computation yields not only  the result that $\{1\}$ is the Groebner Basis of the system but also how we can describe it in terms of the polynomials in the original defining equations. The strategy is to study  these expressions of $1$ in terms of the original polynomials for small dimensions $n$ and from that give an inductive guess for what it will be in higher dimensions. To this end, the explicit equations given in \cite{gorni-gasinska-zampieri} may be useful. We can also work with many specific examples to get the idea. The computation for the mentioned expressions have been implemented in some computer programs (for example Mathematica and Maple), however it seems not available in MuPad.

One difficulty with this approach is that the number of equations defining the Druzkowski matrices is very big, which makes it difficult to check even for small dimensions.  
 
 \subsection{Approach 2: Checking Conditions 1 and 2 on a larger variety} For Druzkowski matrices, Condition 2 is equivalent to the Jacobian conjecture. From the experimental computations above, it seems that the stronger Condition 1 may be also true for Druzkowski matrices. We also know that Condition 1 is true for many other varieties. So if we assume that the Jacobian conjecture is true, then there are varieties strictly containing the Druzkowski matrices on which Condition 2 (and may be also Condition 1) holds. If we can find a "nice" such variety (for example if it is defined by less equations), then we can facilitate the use of computer programs. To guess what should be such nice varieties, we may again  do experiments with small dimensions or specific matrices. The ideal case is that Condition 2 is true for all matrices.  

See the Remark in Section 5.1 for support to this approach.  

 \subsection{Approach 3: Checking Conditions 1 and 2 for integer matrices} Here is another approach, which we feel is also promising. We know that it is sufficient to check either Condition 1 or 2 is true for all integer matrices. This is in general a hard problem belonging to the field of Diophantine equations, but we may go around it in the following way. We can compute the projections from $V_n\cap \{z=Z_k\}$ (or $W_n\cap \{z=Z_k\}$) to the set of all $n\times n$ matrices $\mathcal{M}_n$. Since this image is contained in a strictly smaller subvariety of $\mathcal{M}_n$ (Theorem \ref{Theorem1}),  the closure of the image is also contained in the same strictly smaller subvariety. The ideal defining the latter (which then must be bigger than ${0}$) can be computed explicitly (\cite{cox-little-shea}),  and we can look to see whether it contains some special polynomials $P$ which is always nonzero on rational numbers (in particular, if it is always positive on real numbers). (More generally, we can check whether there is a polynomial $P$ whose zero set is easy to study.) We have searched on randomly generated matrices of various dimensions and ranks, and did not yet find any counterexample to this approach. (Of course, if there are such counterexamples, they must be very rare and would almost never appear when we choose randomly, by results in Section 3.2). 

 \subsection{Approach 4: Checking Condition 1 for a dense subset of Druzkowski matrices} In case Approach 1 does not work (say, simply because we have not enough experimental computations to make a reasonable guess), we may also check Condition 1 on a dense set of Druzkowski matrices as follows. We consider the projection of  $V_n\cap \{z=Z_k\}$ to $\mathcal{M}_n$. This image is not dense in the set of Druzkowski if its closure is not. Then we check that the closure of the complement of this latter set in the set of Druzkowski matrices is the Druzkowski matrices themselves. There are algorithms for doing this using  quotients of ideals \cite{cox-little-shea}. Alternatively, we can proceed in the manner outlined in Approach 0. 

\subsection{Approach 5: Proving inductively on $k$} In the above approaches, we proposed to proving the Jacobian conjecture inductively on either the dimension $n$ or the rank $r$ of the matrix. From Theorems \ref{TheoremFurtherReduction} and \ref{TheoremFurtherReduction1}, there is also one other approach that we can utilise that is of inductively proving on $k$. This approach seems to be quite natural. Right now we know that for $k=1$ or $2$, then the approach works. From the experimental computations, it seems that for $k=3$ the approach also works. How about $k=4$ and higher? Again, we propose that computer programs will be used to help with this. 

\subsection{Approach 6: Checking that $W_n'$ or $W_n$ is connected} This approach is based on Theorem \ref{TheoremEquivalentFormulation}, which says that $JC(\infty )$ is equivalent to the fact that $W_n'$ is connected for all $n\in \mathbb{N}$. In particular, if $W_n$ is connected for all $n\in \mathbb{N}$ then $JC(\infty )$ follows. Connectedness of a variety is a classical subject in Algebraic Geometry, and in principle it is easier to check whether a variety is connected than to compute the Groebner Basis of the ideal defining it. The fact that $W_n$ is defined by only $n+n^2$ equations makes the computations faster. 

\subsection{A related problem} To prove $JC(\infty )$ it suffices to prove for generic Druzkowski matrices. It is hence an important task to be able to generate more Druzkowski matrices so we can test the approaches. If we are able to do so, we can use Theorems \ref{TheoremFurtherReduction} and \ref{TheoremFurtherReduction1} to effectively check. However, this is a difficult problem, as mentioned before. Among the approaches we propose here, the Approaches 2,3 and 6 are not affected by not having a good way to generate Druzkowski matrices, since they concern more general matrices.  

 \subsection{Conclusions} We have given some simplifications to the polynomial systems defining the Druzkowski matrices and the injectivity of the corresponding Druzkowski maps. Based on these, we have proposed some conditions on square matrices which help to solve the Jacobian conjecture and proven that they are true in various cases, with the help of computer programs. The experimental computations show that these simplifications save a lot of computing time in practice, and point out that a stronger Jacobian conjecture seems to be true for Druzkowski maps. The criteria in some of these approaches are shown to be satisfied by a generic Druzkowski map, which is a good sign that our approaches seem very promising. The approaches we proposed are to use computer programs to investigate various special cases to guess an inductive argument. Of course, it may happen that even after we have the data for small dimensions at hand, we still have no clue of how to process it. But even in that case, it is certain that we understand the Jacobian conjecture a little bit more.   
 
It is hopeful that more experimental computations will be achieved. At the moment, only MuPad is available on Tizzard, and it is interactive and sequential, hence not appropriate for parallel computing. A program which can be run in parallel and has more options (for example, the option to express the Groebner Basis in terms of the original defining polynomials seems not available in MuPad) will help to significantly improve  the outcomes of our investigation. Our aim is, for the near future, to at least be able to obtain a proof of the Jacobian conjecture for Druzkowski maps  in small dimensions. Of particular interest are the first unknown cases, like in dimension $9$ or $10$.

Ongoing experimental computations are being performed and the results will be updated when they are available. 

{\bf Remarks.}

1) To show only that the Jacobian conjecture is true for Druzkowski maps of a specific dimension (for example if we want more evidences to support the Jacobian conjecture), we can use other methods instead of that of Groebner Basis. For example, we can simply compute the dimension of the intersection of the variety $V_n\cap \{z=Z_k\}$ and the Druzkowski matrices. For this sole purpose, there are alternative algorithms which may perform faster than computing Groebner Basis, for example:  primary decomposition, the characteristic set method of Reid and Wu, resultants and so on.  However, it does not seem that knowing only that the set is empty will help us in solving the Jacobian in general. It is, in our opinion, the expression of the constant $1$ in terms of the original defining polynomials that is the most essential.  

2) For the case of  Druzkowski maps in dimension $9$, there is an alternative approach toward the Jacobian conjecture. As mentioned in the Introduction, from the results in \cite{yan} it is sufficient to show that the set of Druzkowski matrices having no zero entries on the diagonal is dense in the set of all Druzkowski maps. To this end, there are algorithms to deal with, as mentioned in Section 5.6 above. So we can either try to prove this fact or use computer programs to check. (We note that the claim that the set of matrices whose all diagonal entries are non-zero should be dense in the set of Druzkowski matrices, while very reasonable, is not as obvious as it may seem at first. A similar claim is false: Rusek \cite{rusek} showed that the matrices of corank exactly $1$ is not dense in the set of Druzkowski matrices.) Here is a strategy of using cubic similarity to prove this density claim. We start from a Druzkowski matrix in dimension $\leq 9$, and then proceed to show that two alternative cases happen. Either $A$ is cubic similar to another matrix satisfying the conditions to apply the results in \cite{yan}, or $A$ is so special that it obviously satisfies the Jacobian conjecture.  
 
 \section{Appendix: Mathematica codes}

  \subsection{Mathematica codes for an explicit Druzkowski matrix} In this subsection we include the Mathematica codes to check whether a given Druzkowski matrix satisfies Condition 1 or Condition 2. 
 
 ClearAll["Global`*"]
 
 (*User enters the dimension and the rank*)
 
 n = 
 
r = 

(*User enters the matrix A*)

A=$\{\{\}\}$
 
 A // MatrixForm
 
 MatrixRank[A]
 
 Dimensions[A]
 
 CP = CharacteristicPolynomial[A, t]
 
 (*Variables, I use xx instead of z, and yy instead of y *)
 
 xx = Table[Subscript[x, i], $\{i, n\}$]
 
yy = Table[Subscript[y, i], $\{i, n\}$]
 
(*Compute det [Id +Delta [(sx+ty)*(sx+ty)] .A], and pick out the equations for that to be 1 for all s,t*) 
 
 tt = Table[Subscript[t, i], $\{i, n\}$]
 
ZZ1 = DiagonalMatrix[tt].A

ZZ2 = IdentityMatrix[n] + ZZ1

Detn = -1 + Det[ZZ2] - Det[ZZ1];

SquareLxy = (s*xx + t*yy)*(s*xx+t*yy)

Detxy = Detn /. Table[tt[[i]] $->$ SquareLxy[[i]], $\{i, n\}$]; // Timing

DetA = Det[A]; // Timing

MonomialListDetxy = MonomialList[Detxy, $\{t, s\}$]; // Timing

Length[MonomialListDetxy]

IdealDetxy = MonomialListDetxy /. t $->$ 1 /. s $->$ 1; // Timing

Length[IdealDetxy]

IdealDetxy = Union[IdealDetxy, $\{DetA\}$]; // Timing

Length[IdealDetxy]
 
(*The equations for the injectivity. The first is for Condition 1, the last two are for Condition 2. For Condition 2 we use two different types: the more direct one and the one after transformation*)

Equations = xx + A.(xx*xx*xx + xx*yy*yy)

Equations1 = xx + (A.xx)*(A.xx)*(A.xx) + (A.xx)*(A.yy)*(A.yy)

Equations2 = xx + A.(xx*xx*xx) + A.(xx*(A.yy)*(A.yy))

(*Total equations. For Condition 1, we need to add the IdealDetxy. For Condition 2, IdealDetxy is automatically satisfied if we start with a Druzkowski matrix. If A is not Druzkowski, we need to add IdealDetxy in.*)

TotalEquations = Union[IdealDetxy, Equations]; // Timing

Length[TotalEquations]

TotalEquations1 = Equations1; // Timing

Length[TotalEquations1]

TotalEquations2 = Equations2; // Timing

Length[TotalEquations1]

(*Find Groebner Bases to check Conditions 1 and 2. Look to see x1,\ldots, xn in the bases.*)

GB = GroebnerBasis[TotalEquations, Union[xx, yy]]; // Timing

Print[GB]

GB1 = GroebnerBasis[TotalEquations1, Union[xx, yy]]; // Timing

Print[GB1]

GB 2= GroebnerBasis[TotalEquations2, Union[xx, yy]]; // Timing

Print[GB2]

(*Check Theorem  \ref{TheoremFurtherReduction1}. Look to see the Groebner Basis is $\{1\}$. User enters a value of k between 3 and n, but  this can be also done automatically*)
 
 (*User enters value for k*)
 
 k=
 
 xS = Table[0, $\{i, n\}$]
 
 Do[xS[[i]] = 1, $\{i, k\}$]
 
Print[xS]
 
GB3 = GroebnerBasis[Union[TotalEquations /. Table[xx[[i]] $->$ xS[[i]], $\{i, n\}$],$\{\}$], Union[xx,yy]]; // Timing

Print[GB3]

(*If we want to describe the Groebner Basis in terms of the original defining polynomials we use instead the following codes. Then we printout the equations in this case to check.*)

$\{$GB3,cofactors3 $\}$ = GroebnerBasis`BasisAndConversionMatrix[Union[TotalEquations /. Table[xx[[i]] $->$ xS[[i]], $\{i, n\}$],$\{\}$], Union[xx,yy]]; // Timing

Print[GB3]

Print[cofactor3]

TotalEquations /. Table[xx[[i]] $->$ xS[[i]], $\{i, n\}$]
 
 (*Automatic code*)
 
 xS=Table[0,$\{i,n\},\{j,n\}$]
 
 Do[xS[[i,j]]=1,$\{i,n\},\{j,i\}$]
 
 GB4 = Table[List[], $\{i,n\}$]

Do[GB4[[i]] = GroebnerBasis[TotalEquations /. Table[xx[[j]] $->$ xS[[i, j]], $\{j, n\}$],Union[xx, yy]], $\{i, n\}$]; // Timing

Print[GB4]
 
  \subsection{Mathematica codes for a general Druzkowski matrix} The Mathematica codes for this case, while similar to that for a specific matrix, are different in several aspects:
  
a) First, the matrix A is not entered by a user but is implemented automatically. 

FirstRows = Table[Subscript[a, i, j], $\{i, r\}$, $\{j, n\}$]

Coefficients = Table[Subscript[b, i, j], $\{i, n - r\}$, $\{j, r\}$]

A = Table[0, $\{i, n\}, \{j, n\}$]

Do[A[[i]] = FirstRows[[i]], $\{i, r\}$]

Do[A[[r + i]] = Sum[Coefficients[[i, j]]*FirstRows[[j]], $\{j, r\}], \{i, n - r\}$]

Print[A]

Consequently, we need to collect the variables to be used in the Groebner Basis computation. 

ListA = List[]

Do[ListA = Union[ListA, $\{$FirstRows[[i, j]]$\}$], $\{i, r\}, \{j, n\}$]

Do[ListA = Union[ListA, $\{$Coefficients[[i, j]]$\}$], $\{i, n - r\}, \{j, r\}$]

Length[ListA]

Print[ListA]

b) Second, we need to add in the equations for the Druzkowski maps: det (Id +Delta [(A.x)*(A.x)].A)=1 for all $x$.

SquareLxy1 = (A.xx)*(A.xx)

Detxy1 = -1 + Detn /. Table[tt[[i]] $->$ SquareLxy1[[i]], $\{i, n\}$]; // Timing
 
 MonomialListDetxy1 = MonomialList[Detxy1, xx]; // Timing
 
 IdealDetxy1 = MonomialListDetxy1 /. Table[xx[[i]] $->$ 1, $\{i, n\}$]; // Timing
 
 c) Third, in the computation of Groebner Bases, we add ListA into the list of variables Union[xx,yy].

\end{document}